\documentclass{amsart}
\usepackage[foot]{amsaddr}
\usepackage{amsmath}
\usepackage{amsfonts}
\usepackage{amssymb}
\usepackage{bm}
\usepackage{cancel}
\usepackage{color}
\usepackage{comment}
\usepackage{enumerate}
\usepackage{etoolbox}
\usepackage{graphicx}
\usepackage{soul}
\usepackage{stmaryrd}
\usepackage{xcolor}
\usepackage{clayams}

\newcommand\numberthis{\addtocounter{equation}{1}\tag{\theequation}}

\DeclareMathOperator*{\argmin}{arg\,min}

\graphicspath{{./figs/}}
\title[\sf Strong convergence in forward-backward splitting]{On the strong convergence of forward-backward splitting in reconstructing jointly sparse signals}

\author[\sf N.~Dexter]{{\sf Nick~Dexter}$^\star$}
\address{$^\star$\sf Department of Mathematics, Simon Fraser University, Burnaby, BC V5A 1S6, Canada ({\tt nicholas\_dexter@sfu.ca}).}

\author[\sf H.~Tran]{\sf Hoang~Tran$^\dagger$}
\address{$^\dagger$\sf Department of Computational and Applied Mathematics, Oak Ridge National Laboratory \\
Oak Ridge TN 37831-6164 ({\tt tranha@ornl.gov}).}

\author[\sf C.~G.~Webster]{\sf Clayton G.~Webster$^\ddag$}
\address{$^\ddag$\sf 
Oden Institute for Computational Engineering \& Sciences The University of Texas at Austin,
Austin, Texas 78712, and 
Behavioral Reinforcement Learning Lab, Lirio LLC., Knoxville, TN 37923
({\tt claytongwebster@utexas.edu}).}

\date{\today}

\begin{document}

\maketitle
\begin{abstract}
We consider the problem of reconstructing an {\em infinite} set of sparse, finite-dimensional vectors, that share a common sparsity pattern, from incomplete measurements. This is in contrast to the work \cite{Daubechies2004}, where the single vector signal can be infinite-dimensional, and \cite{FornasierRauhut08}, which extends the aforementioned work to the joint sparse recovery of \textit{finite} number of {infinite}-dimensional vectors.  In our case, to take account of the joint sparsity and promote the coupling of nonvanishing components, we employ a convex relaxation approach with mixed norm penalty $\ell_{2,1}$.
This paper discusses the computation of the solutions of linear inverse problems with such relaxation by a forward-backward splitting algorithm.   However, since the solution matrix possesses infinitely many columns, the arguments of \cite{Daubechies2004} no longer apply.  As such, we establish new strong convergence results for the algorithm, in particular when the set of jointly sparse vectors is infinite. 
\end{abstract}

%============================
\section{Introduction}
\label{sec:intro}

 Sparse signal reconstruction seeks to solve many ill-posed problems arising in source separation, denoising, and compressed sensing \cite{CRT06,Donoho06} by exploiting the additional sparsity constraint. In the basic model, the signal i
 s an unknown vector ${\bm c} \in \mathbb{R}^{N}$, and the sensing process yields a measurement vector ${\bm u} \in \mathbb{R}^m$ that is formed by the product of ${\bm c}$ with a sensing matrix, i.e., ${\bm u} = {\bm A} {\bm c}$, where ${\bm A} \in \mathbb{R}^{m\times N}$. The key observation is that when the signal ${\bm c}$ is sufficiently sparse, it can still be uniquely determined from an underdetermined set of measurements ($m< N$). To overcome the NP-hardness of directly finding the sparsest ${\bm c}$ consistent with a given measurement, various greedy and convex relaxation strategies have been proposed and demonstrated, both empirically and theoretically, to have good reconstruction performance in a range of settings. 

In parallel to developments in sparse signal models, many application scenarios have motivated research interest in processing not just a single signal, but many signals or channels at the same time. In such scenarios, these signals not only possess sparse representations individually, but can also share common sparsity patterns. The problem of simultaneous recovery of jointly sparse signals from incomplete measurements has been referred to as multichannel sparse recovery, joint sparse recovery, simultaneous sparse approximation or the multiple measurement vector (MMV) problem \cite{CREK05,ChenHuo06,TroppGilbertStrauss06,Tropp06,FornasierRauhut08,MishaliEldar08,GRSV08,EldarRauhut10,BergFriedlander10,DaviesEldar12,LBJ12}. Some well-known applications can be found in, for instance, neuroimaging \cite{GGR95,GorodnitskyRao97,PLM97}, DNA microarrays \cite{EricksonSabatti05,PVMH08}, and sensor networks \cite{BaronWakin05proc,WakinSarvotham05}. Recently, the joint sparse recovery problem also arises in the approximation of parameterized partial differential equations (PDEs) modeling physical systems with uncertain inputs. As this application has not drawn much attention so far, we include a detailed discussion in Section \ref{subsec:parameterized_PDE}. 

This paper is concerned with the simultaneous recovery of a collection of sparse signals $\{{\bm c}^{(r)}\in \mathbb{R}^N: r\in \mathbb{N}\}$, given infinitely many measurements 
\begin{align}
\label{intro:eq1}
{\bm u}^{(r)} = {\bm A} {\bm c}^{(r)} + {\bm e}^{(r)}, \ \ \forall r \in \mathbb{N}. 
\end{align}
Here, ${\bm u}^{(r)} \in \mathbb{R}^m$ are the measurement vectors, ${\bm A} \in \mathbb{R}^{m\times N}$ is a predefined sampling matrix and ${\bm e}^{(r)} \in \mathbb{R}^m$ are unknown noise vectors. Grouping separate signals, measurements, and noise vectors into the matrices 
\begin{align*}
{\bm c} &= [{\bm c}^{(1)} \ {\bm c}^{(2)} \  \ldots \ {\bm c}^{(r)}  \ldots \ ] \in \mathbb{R}^{N \times \mathbb{N}},
\\
{\bm u} &= [{\bm u}^{(1)} \ {\bm u}^{(2)} \ldots \ {\bm u}^{(r)} \ldots \ ]  \in \mathbb{R}^{m\times \mathbb{N}}, \ \ 
{\bm e} = [{\bm e}^{(1)} \ {\bm e}^{(2)} \ \ldots \ {\bm e}^{(r)} \ldots \ ] \in \mathbb{R}^{m\times \mathbb{N}}, 
\end{align*}
the considered problem becomes the reconstruction of the signal matrix ${\bm c}$ from 
\begin{align}
\label{intro:eq2}
{\bm u} = {\bm A} {\bm c} + {\bm e}. 
\end{align}
We assume the signals ${\bm c}^{(r)}$ possess a joint sparsity pattern, under which all except a few rows of ${\bm c}$ are negligible. Then, our goal is to find a sparse matrix solution consistent with \eqref{intro:eq2}.  

For a real matrix ${{\bm x}}$, let us denote by ${{\bm x}}_{i}$ and ${{\bm x}}^{(r)}$ the $i$-th row and $r$-th column of ${{\bm x}}$, respectively. We consider the reconstruction of a sparse matrix ${\bm c}$, given measurement ${\bm u}$ satisfying \eqref{intro:eq2}, via ${\bm x}^{*}$ solving the convex minimization program 
\begin{align}
\label{intro:l_21}
{\bm x}^{*}  = \argmin_{{\bm x} \in \mathbb{R}^{N\times\mathbb{N}}} \|{\bm x}\|_{2,1} + \frac{\mu}{2} \|{\bm A} {\bm x} - {\bm u}\|^2_{2,2}, 
\end{align}
where the matrix norm $\|\cdot\|_{p,q}$ is defined as  
$
\|{{\bm x}}\|_{p,q} = \left(\sum_{i} \|{{\bm x}}_{i}\|_p^q\right)^{1/q},\, \forall 0 < p,q < \infty.
$
Use of the mixed norm $\ell_{2,1}$ in regularization, equivalent to first finding the $\ell_2$ norm (promoting nonsparsity) to rows and then applying the $\ell_1$ norm (promoting sparsity) to the resulting vector, has been known a tractable and efficient approach to recover signal matrix with fewest nonzero rows, just as the $\ell_1$ regularization relaxes and replaces the $\ell_0$ minimization in single sparse approximation. 

This paper studies and analyzes a forward-backward splitting approach to solve \eqref{intro:l_21}. In this context, it is also known as the proximal Landweber method, { the thresholded Landweber method, or the iterative soft-thresholding algorithm (ISTA)}. This approach has been studied intensively in the literature for solving standard, unconstrained $\ell_1$ minimization { or related sparsity-promoting regularized} problems motivated by single signal recovery, and convergence theory exists in many previous works, see, e.g., {\cite{Daubechies2004,Combettes2004,Combettes2005,HYZ08,doi:10.1137/060669498,nutini2018activeset, attouch_convergence_2013,liang_activity_2017,fadili_sensitivity_2018}}. The extension of forward-backward splitting to joint sparse recovery is also quite straightforward. It is indeed possible to obtain many interesting results for joint sparse recovery just by slightly modifying available proofs for single-sparse vector recovery (that is, replacing signal vectors by signal matrices, and the absolute value of vector components by the norm of matrix rows). 
Yet, some important questions still remain open. To date, most of the strong convergence results (for single or simultaneous reconstructions) rely on either the strict convexity of the {fidelity} term or the finite dimension of the signals. 
Some notable exceptions are \cite{Daubechies2004}, where the single vector signal can be infinite-dimensional, and \cite{FornasierRauhut08}, which includes, among others, an extension of the aforementioned work to the joint sparse recovery of \textit{finite} number of {infinite}-dimensional vectors. 
The infinite dimensional setting has also been analyzed recently in \cite{garrigos2020thresholding}. This work however assumes separable regularizer, therefore is more in line with \cite{Daubechies2004} and not applicable to joint sparse recovery.

In this manuscript, our main contribution is a strong convergence 
result for the forward-backward splitting algorithm in a joint sparse recovery scenario where neither the strict convexity nor finite dimensionality assumptions holds. In particular, we consider the simultaneous reconstruction in which the dimension of each signal is finite, but the number of signals is \textit{infinite}. This setting is of mathematical interest and arises in compressed sensing-based approximation of parametric PDEs, discussed in Section \ref{subsec:parameterized_PDE} below. 
Here, as the solution matrix possesses infinitely many columns, the arguments of \cite{Daubechies2004,garrigos2020thresholding} are no longer applicable and our analysis needs to follow a new, completely different path.

\subsection{A motivating example: Parameterized elliptic PDEs}

\label{subsec:parameterized_PDE}

The problem of joint sparse recovery arises in, among others, the approximation of high-dimensional parameterized systems. In these contexts, the target quantities of interest are often associated with the solution of a parameterized PDE of the form: find $u(\cdot,\bm{y}) : \overline{D} \times \mathcal{U} \to \mathbb{R}$ for all ${\bm y} \in \mathcal{U}$ such that  
\begin{equation}
\label{eq:genPDE}
\mathcal{L}(u(\cdot, \bm{y}),\bm{y})=0, \ \ \text{ in }D, 
\end{equation}
where $\mathcal{L}$ is a differential operator defined on a spatio-temporal domain $D$ and $\bm{y}$ is a parameter vector in a high-dimensional tensor product domain $\mathcal{U} \subset \mathbb{R}^d$. One typical goal is to simultaneously approximate the entire parametric solution map $\bm{y} \mapsto u(\cdot, \bm{y}) \in {\mathcal{V}}$ up to a prescribed accuracy with minimal computational cost, where ${\mathcal{V}}$ is the solution space, typically a separable Hilbert space. As this solution map is now well-known to be smooth for a wide class of parameterized PDEs, global polynomial approximation is an appealing approach to solve \eqref{eq:genPDE}. 

Define the index set $[N]:=\{1,2,\ldots, N\}$, and let $\{\Psi_j\}_{j=1}^N$ be fixed orthonormal polynomial basis over $\mathcal{U}$. Let $u_{N} = \sum_{j =1}^N c_{j}(\cdot) \Psi_j({\bm y})$ be the projection of $u$ to the space ${\mathcal{V}} \otimes \text{span}\{\Psi_{j}: j\in[N]\}$. We seek to construct an approximation $u^{*}$ to $u_{N}$ (and thus, to the solution $u$) of the form:
\begin{align}
\label{poly_expans}
u^{*}(\cdot,{\bm y}) = \sum_{j=1}^N c^{*}_{j}(\cdot) \Psi_j({\bm y}),
\end{align}
where $\{c^{*}_{j}\}_{j =1}^N \subset {{\mathcal{V}}}$ are the Hilbert-valued coefficients to be computed. Compressed sensing-based polynomial approximations \cite{RauWard12,Adcock15b,ChkifaDexterTranWebster15} allows the truncation of the expansion \eqref{poly_expans} in a large, not necessarily optimal index set. One first generates $m$ samples ${\bm y}_1, \ldots, {\bm y}_m$ in ${\mathcal{U}}$ independently from the orthogonalization measure associated with $\{\Psi_j\}_{j=1}^N$ and solves the equation \eqref{eq:genPDE} at these samples, to form the normalized measurement matrix and normalized measurement vector
$$
\bm{A} := \left(\frac{1}{\sqrt{m}} \Psi_{j}({\bm y}_i)\right)_{{{1\le i \le m},\, {1\le j \le N}}}
\qquad
{\bm u} := \left(\frac{1}{\sqrt{m}} u(\cdot, {\bm y}_i)\right)_{1\le i \le m}.
$$

Taking into account that the true unknown coefficient ${\bm c} = (c_{j})_{j =1}^N$ approximately solves the linear system 
\begin{align}
\label{und_system}
{\bm u} = \bm{A} {\bm x}, \ \ \ {\bm x}\in {\mathcal{V}}^N,
\end{align}
and further, under reasonable assumptions on $\mathcal{L}$ from \eqref{eq:genPDE} the sequence $\{c_{j}\}_{j=1}^N$ decays fast to $0$ in ${\mathcal{V}}$-norm with a large percentage of its elements being negligible \cite{CDS11,CCS14,TranWebsterZhang14}, it is reasonable to approximate ${\bm c}$ by ${\bm x}^*$, the solution to the regularized problem:
\begin{align}
\label{eq:lV1}
{\bm x}^{*} = \argmin_{{\bm x} \in {{\mathcal{V}}}^N} \|{\bm x}\|_{{\mathcal{V}},1} + \frac{\mu}{2}  \|\bm{A} {\bm x} - {\bm u}\|^2_{{\mathcal{V}},2},
\end{align}
in which the number of samples ($m$) can be significantly less than the size of polynomial subspace ($N$). Here, the norm $\|\cdot\|_{{\mathcal{V}},q}$ is defined for ${\bm c}\in {\mathcal{V}}^N$ as $\|{{\bm c}}\|_{{\mathcal{V}},q} := (\sum_{j=1}^N \|{c}_{j}\|^q_{{\mathcal{V}}})^{1/q}$. This is arguably the most natural extension of the $\ell_1$ minimization approach, traditionally for real and complex signal recovery, to the reconstruction of sparse generalized vectors, each component of which is Hilbert-valued. In-depth analysis and application of \eqref{eq:lV1} in solving parameterized PDEs are conducted in \cite{DexterTranWebster17}.

Let $\{\phi_r\}_{r\in \mathbb{N}}$ be an orthonormal basis of ${\mathcal{V}}$, then $c_{j} \in {\mathcal{V}}$ is uniquely represented as 
$$
c_j  = \sum_{r\in \mathbb{N}} c_{j,r} \phi_r,\ \ \text{with } \ c_{j,r} \in \mathbb{R}. 
$$
Each coefficient $c_{j}$ corresponds to an $\mathbb{R}^{1\times \mathbb{N}}$ vector $(c_{j,1},c_{j,2},\ldots, c_{j,r}, \ldots)^{\top}$, thus, ${\bm c} = (c_{j})_{j=1}^N$ is completely determined by the $\mathbb{R}^{N\times \mathbb{N}}$ matrix ${\widehat{\bm{c}}} = (c_{j,r})_{1\le j\le N,r\in \mathbb{N}}$. 
Furthermore, by Parseval's identity we have $\|\bm{c}\|_{\mathcal{V},q} \equiv \| \widehat{\bm{c}} \|_{2,q}$, so that we may replace $\mathcal{V}$ with $\ell_2$ and identify $\bm{c}$ with $\widehat{\bm{c}}$.
In practice, one may choose to work in a subspace $\mathcal{V}_h$ of $\mathcal{V}$, e.g., conforming finite element spatial discretizations of \eqref{eq:genPDE}.
The convergence results established in this manuscript for reconstructing infinitely many jointly sparse vectors reveal that forward-backward splitting approach enjoys strong convergence in solving the abstract problem \eqref{eq:lV1}, even before such spatial (or temporal) discretization is introduced.

\subsection{Related works}

The study of the convergence properties of the forward-backward algorithm over infinite-dimensional function spaces has a long history, with work by R.~T. Rockafellar, R. J.-B. Wets, and J.-J. Moreau dating back to the 1960's. 
An excellent summary of these results can be found in the commentary of \cite[Chapter 1]{rockafellar2009variational}.
A key challenge in the analysis of the algorithm in such settings is the lack of compactness of infinite dimensional bounded and closed sets, see \cite[Chapters 3 \& 6]{brezis2010functional} for more discussion. Weak convergence may be established in such settings using standard approaches, however strong convergence often requires additional assumptions on the problem formulation. 

There have been several approaches for the joint sparse recovery, many of which are extensions from the single signal recovery. These include greedy methods \cite{CREK05,TroppGilbertStrauss06,GRSV08}, and algorithms based on mixed norm optimization \cite{CREK05,Tropp06,FornasierRauhut08,EldarRauhut10,BergFriedlander10}. A method to reduce the multiple measurement problem to the basic model of a single sparse vector via a random projection that preserves the sparsity structure was proposed in \cite{MishaliEldar08}. It is also possible to improve the joint sparse recovery by exploiting the rank of signal matrix, see \cite{DaviesEldar12}, as well as by formulating and solving the problem as a nonconvex optimization problem on manifolds, \cite{PetrosyanTranWebster18}.    

The properties and computation of the solutions of problem \eqref{intro:l_21} (and related variants) have been studied in several previous works. For instance, a FOCUSS algorithm was developed in \cite{CREK05} for $\ell_{2,q}$ penalty with $q\in (0,1]$, and shown by numerical tests to converge to a sparse solution. \cite{Tropp06} established sufficient conditions under which the $\ell_{\infty,1}$ regularization computes sparse solutions to simultaneous approximation problems. In \cite{FornasierRauhut08}, a double-minimization scheme is proposed to model joint sparsity, the first step of which is a forward-backward splitting algorithm to solve $\ell_{p,1}$ regularized problem ($1\le p \le \infty$). Strong convergence of this algorithm was proved. Further progresses in this topics were made in \cite{BergFriedlander10,EldarRauhut10}, where the benefit of simultaneous reconstruction with mixed norm $\ell_{2,1}$ over sequential reconstruction of each signal vectors was analyzed.   

A problem closely related to joint sparse recovery is block sparse recovery, \cite{EldarMishali09,SPH09,EKB10,BCDH10}, where the output vectors are acquired via different sampling matrices, i.e., ${\bm A}$ is replaced by ${\bm A}^{(r)}$ in \eqref{intro:eq1}. While this problem is not discussed in detail herein, we expect that our convergence result can be extended to the block sparse setting with slight modifications. Finally, sparse representation over infinite spaces has also been studied in the context of kernel methods, see \cite{koppel2019parsimonious}.

\subsection{Organization} Our paper is organized as follows. In Section \ref{sec:setting}, we describe the forward-backward splitting scheme for joint sparse recovery and provide necessary background properties. Our main convergence results will be presented in Section \ref{sec:algorithms}. The concluding remarks can be found in Section \ref{sec:conclusion}. 
%============================

%============================
\section{Forward-backward splitting algorithm for joint sparse recovery}
\label{sec:setting}

In this section, we present a forward-backward splitting algorithm for solving the optimization problem \eqref{intro:l_21}, which can be written in detail as
\begin{equation}
\label{eq:optimization_problem}
\min_{\bm x = (\bm x_1,\ldots,\bm x_N)\in\mathbb{R}^{N\times \mathbb{N}}} \sum_{i=1}^N \|\bm x_i\|_2 + \frac{\mu}{2} \sum_{i=1}^m \| (\bm A \bm x)_i - \bm u_i \|_2^2,
\end{equation}
where $\mu>0$, $\bm A$ is an $m\times N$ matrix, $\bm u=(\bm u_1,\ldots,\bm u_m)\in \mathbb{R}^{m\times\mathbb{N}}$ and $i$ denotes the index of the matrix rows.
For simplicity, we assume $\mu = 1$ in \eqref{eq:optimization_problem}, noting that all of the analysis to follow holds in the case of arbitrary $\mu$. Let us define $\mathcal{H} = \{{\bm x} \in \mathbb{R}^{N\times \mathbb{N}} : \|{\bm x}\|_{2,2} < \infty\}$, we also assume ${\bm u} \in {\mathcal{H}}$, so that \eqref{eq:optimization_problem} has solutions in ${\mathcal{H}}$. Let 
$$
\phi_1({\bm x}) = \|{\bm x}\|_{2,1}\text{ and }\phi_2({\bm x}) = \frac{1}{2} \| {\bm A} {\bm x} - {\bm u} \|_{2,2}^2.
$$
Then $\phi = \phi_1 + \phi_2$ represents a splitting of the objective of \eqref{eq:optimization_problem} into the non-differentiable and Fr\'{e}chet differentiable parts, $\phi_1$ and $\phi_2$, respectively.
Define $T_1 = \partial \phi_1$, $T_2 = \partial \phi_2 = \{\nabla \phi_2\}$, and $T = \partial \phi = T_1 + T_2$,  
 the solutions ${\bm x}^* \in {\mathcal{H}}$ of \eqref{eq:optimization_problem} are characterized by
\begin{align}
\label{eq:X_star}
\bm{0} \in \partial\phi_1({\bm x}^*) + \{\nabla \phi_2({\bm x}^*)\},\ \ \text{ or }\ \ \bm{0} \in (T_1+T_2)({\bm x}^*),
\end{align}
where $\partial \phi_1$ represents the subdifferential of $\phi_1$: for all ${\bm x}\in {\mathcal{H}}$,
\begin{align}
\label{eq:phi_1_subdifferential}
\partial \phi_1({\bm x}) = \{{\bm v}\in {\mathcal{H}} : \langle {\bm v}' - {\bm x}, {\bm v} \rangle_{2,2} + \phi_1({\bm x}) \le \phi_1({\bm v}'), \;\; \forall {\bm v}'\in {\mathcal{H}} \}.
\end{align} 
Let $X^* := \{{\bm x}\in {\mathcal{H}}: \bm{0} \in (T_1+T_2)({\bm x}) \}$, we aim to find ${\bm x}^* \in X^*$ via a formulation of the forward-backward splitting algorithm \cite{Bruck1975,Goldstein1964,Chen1997}, which makes use of the splitting of $T$ into $T_1$ and $T_2$, derived in the setting of joint-sparse recovery.
The forward-backward algorithm is a two-step fixed-point algorithm that involves an explicit (forward) step composed with an implicit (backward) step.
It is efficient, in that it only involves alternating steps requiring relatively cheap computations, using the functions $T_1$ and $T_2$ separately.
In this way it avoids direct computation of $(T_1+T_2)^{-1}(\bm{0})$, which may not be feasible. 

The algorithm can be derived as follows. Let $\tau > 0$, then from \eqref{eq:X_star} we have
\begin{align*}
\bm{0} \in T({\bm x}) 
                  & \iff {\bm x} = (I + \tau T_1)^{-1} (I - \tau T_2) {\bm x}. \numberthis \label{eq:forward_backward_derivation}
\end{align*}
\eqref{eq:forward_backward_derivation} is well-defined, since, as we shall see, $(I-\tau T_2)$ is single-valued and $(I+\tau T_1)$ is invertible. The last identity in \eqref{eq:forward_backward_derivation} leads to the {\em forward-backward splitting} algorithm: 
given initial guess ${\bm x}^{0}\in {\mathcal{H}}$, compute
\begin{equation}
\label{eq:forward_backward_iteration}
{\bm x}^{k+1} = (I + \tau T_1)^{-1} (I - \tau T_2) {\bm x}^k, 
\end{equation}
where ${\bm x}^{k}$ denotes the approximation at $k$-th iterate. We note that the forward-backward algorithm derived here differs from those in many previous works in that it applies soft-thresholding to a target signal whose components can be elements of an infinite-dimensional space, for example, Hilbert-valued (see also Section \ref{subsec:parameterized_PDE}). As opposed, previous derivations only consider a single or a finite collection of real or complex-valued signals. Certainly, one can consider to extend other well-known sparse and joint sparse algorithms, such as Douglas-Rachford splitting \cite{Combettes2004} and  Alternating Direction Method \cite{QinGoldfarb12, 10.1117/12.2024410}, to this scenario. We leave the derivation and analysis of such extensions to future works.

Let us define 
$$
J_\tau := (I + \tau T_1)^{-1},\ \ G_\tau := (I - \tau T_2),\ \ \text{ and }\ S_\tau := J_\tau \circ G_\tau,
$$
then \eqref{eq:forward_backward_iteration} can be written as ${\bm x}^{k+1}  = S_\tau ({\bm x}^{k})$, and it is clear from \eqref{eq:X_star} and \eqref{eq:forward_backward_derivation} that $X^*$ is the set of fixed points of $S_\tau$. 

It is standard to derive the following formulations for $G_{\tau}$ and $J_{\tau}$:
\begin{align}
\label{eq:G_tau}
G_\tau({\bm x}) & = {\bm x} - \tau {\bm A}^*({\bm A}{\bm x} - {\bm u}), \\
\label{eq:J_tau}
(J_{\tau}({\bm x}))_j & = \frac{{\bm x}_{j}}{\|{\bm x}_{j}\|_2} \cdot \max \{ \|{\bm x}_{j} \|_2 - \tau, 0\}, \qquad 1\le j\le N. 
\end{align}
One can observe that the forward operator $G_\tau$ resembles a step of gradient descent algorithm with stepsize $\tau$ for minimizing $\phi_2$. The backward operator $J_{\tau}$, on the other hand, is a soft  thresholding step associated with proximal point method. As a result, algorithm \eqref{eq:forward_backward_iteration} can be considered as an instance of the {\em proximal-gradient} method. 

Our analysis is conducted under the following assumption, which states that ${\bm A}^*{\bm A}$ (the Hessian of $\phi_2$) has bounded spectral norm, and that the step size $\tau$ is chosen appropriately with respect to its spectral radius. 
\begin{assumption}
\label{ass:tau_bounds}
Let ${\bm H} := {\bm A}^*{\bm A}$ and $\|{\bm H}\|_2 < +\infty$, we choose the step size $\tau$ in \eqref{eq:forward_backward_iteration} satisfying $0 < \tau < 2/\|{\bm H}\|_2$.
\end{assumption}

Under this assumption, we observe that $G_\tau$ is nonexpansive, i.e., satisfies
\begin{align}
\label{eq:G_tau_FNE}
\| G_\tau ({\bm v}) - G_\tau ({\bm w}) \|_{2,2} \le \| {\bm v} - {\bm w} \|_{2,2} \qquad \forall {\bm v},{\bm w}\in {\mathcal{H}},
\end{align}
Moreover, following arguments from \cite[Chapter 4]{Bauschke2011}, one can show that $J_\tau$ is row-wise firmly nonexpansive, i.e.,
  \begin{gather}
  \label{eq:J_tauj_FNE}
\begin{aligned}
 \| (J_{\tau} ({\bm v}))_j - (J_{\tau} ({\bm w}))_j \|_{2}^2 \le \| {\bm v}_j - {\bm w}_j \|_{2}^2 - \| ((I - J_{\tau}) {\bm v})_j & - ((I - J_{\tau}) {\bm w})_j \|_{2}^2,
\\
&  \forall {\bm v}, {\bm w} \in {\mathcal{H}},\, \forall j\in [N]. 
\end{aligned}
\end{gather}
These properties are essential for our convergence proofs. We remark that if further assumptions are imposed on ${\bm H}$ to make $G_{\tau}$ a contraction on the whole (or certain subspaces of) ${\mathcal{H}}$, the desired strong convergence can be obtained via a routine manner from classical theory (see discussion in Section \ref{sec:lin_convergence}). It is also possible to prove strong convergence in joint sparse recovery with the nonexpansiveness properties \eqref{eq:G_tau_FNE} and \eqref{eq:J_tauj_FNE} by slightly extending the available arguments in single vector recovery, e.g., \cite{HYZ08}, given the set of measurements being finite, i.e., having $K<\infty$ measurements (in which case $\mathbb{R}^{N\times K}$ is locally compact). However, we stress that our below analysis requires neither of these additional assumptions.  

%============================
\section{Convergence results}
\label{sec:algorithms}
In this section, we present our main results 
showing that the sequence $\{{\bm x}^{k}\}$ obtained by iterating \eqref{eq:forward_backward_iteration} converges strongly to an element ${\bm x}^*\in X^*$ from any initial guess ${\bm x}^{0} \in {\mathcal{H}}$. Denote ${\mathcal{H}}^{0} = \{{\bm z} \in \mathbb{R}^{\mathbb{N}}: \|{\bm z}\|_2 <\infty \}$, and let ${\mathcal{P}}_\tau$ be the metric projection from ${\mathcal{H}}^0$ onto $B_2(\bm{0},\tau)$, i.e., for ${\bm z}\in {\mathcal{H}}^0$,
\begin{align}
\label{eq:tau_ball_metric_projection}
{\mathcal{P}}_\tau ({\bm z}) = \begin{cases}
\tau \frac{{\bm z}}{\|{\bm z}\|_{2}}, & {\bm z} \not \in B_{2}(\bm{0},\tau), \\
{\bm z},                            & {\bm z} \in B_{2}(\bm{0},\tau),
\end{cases}
\end{align}
so that $J_{\tau}$ can be represented row-wise as 
$
(J_{\tau}({\bm x}))_j = (I - {\mathcal{P}}_{\tau}) ({\bm x}_j)$ for $\bm{x}\in\mathcal{H}$.

Our analysis relies on the following partition of the index set $[N]$, inspired by \cite[Definition 4.3]{HYZ08}, for the joint-sparse recovery setting.
\begin{definition}
\label{def:L_E_definitions}
For ${\bm x}^* \in X^*$, we define
\begin{align*}
L & := \{ j\in [N] : \| (\nabla \phi_{2}({\bm x}^*))_j \|_{2} < 1 \}, \quad E:= \{ j\in [N]: \| (\nabla \phi_{2}({\bm x}^*))_j \|_{2} = 1 \}, \numberthis \label{eq:L_E_definitions} \\
\omega & := \min_{j\in L} \ \tau(1 - \| (\nabla \phi_{2}({\bm x}^*))_j \|_{2}) . \numberthis \label{eq:omega_def}
\end{align*}
\end{definition}
The intuition for this partition can be derived as follows.
It is easy to see from the definition of $X^*$ and the subdifferential of $\phi_1$ that
\begin{align*}
{\bm x}^* \in X^* 
\iff (\nabla \phi_{2}({\bm x}^*))_j \in 
\begin{cases}
\{-{\bm x}^*_j/\|{\bm x}^*_j\|_{2}\}, & {\bm x}^*_j \neq \bm{0}, \\
B_{2}(\bm{0},1),           & {\bm x}^*_j = \bm{0} ,
\end{cases}
\quad \forall j\in [N],
\end{align*}
implying that $\| (\nabla \phi_{2}({\bm x}^*))_j \|_2 \le 1$ for all $j\in [N]$ and ${\bm x}^*\in X^*$.
Hence $L\cup E = [N]$, and 
\begin{align}
\label{eq:supp_of_solns}
\textnormal{supp}({\bm x}^*) \subseteq E,\;\; \textnormal{ and } \;\; L \subset (\textnormal{supp}({\bm x}^*))^c.
\end{align}
Furthermore, if ${\bm x}^* \in X^*$ and $j\in L$, 
\begin{align*}
\tau(1 - \| (\nabla \phi_{2}({\bm x}^*))_j \|_2) 
    = \tau - \|{\bm x}_j^* - \tau (\nabla \phi_{2}({\bm x}^*))_j \|_2 = \tau - \| (G_{\tau} ({\bm x}^*))_j \|_2,
\end{align*}
therefore  

\begin{equation}
\label{eq:omega_alt_def}
\omega =  \min_{j \in L} \,\left(\tau - \| (G_{\tau} ({\bm x}^*))_j \|_2\right) > 0.
\end{equation}
Using this partition of $[N]$, our strong convergence result is obtained in three steps:

\begin{enumerate}

\item Establish finite convergence on the set $L$, i.e., determine a bound on the number of iterations $K$ such that ${\bm x}^{k}_j = {\bm x}^*_j = \bm{0}$ for each $j\in L$ whenever $k\ge K$. Therefore, in light of \eqref{eq:supp_of_solns}, the iterations \eqref{eq:forward_backward_iteration} partially identify the support of an element ${\bm x}^*\in X^*$ after a finite number of iterations (Lemma \ref{lem:finite_convergence}),

\item Establish angular convergence: $\theta_j^{k}\to 0$ as $k\to\infty$, {where, for $j\in \textnormal{supp}({\bm x}^*) \subseteq E$, $\theta_j^k$ is the angle between ${\bm x}^{k}_j$ and ${\bm x}^*_j$, and, for $j\in E\setminus \textnormal{supp}({\bm x}^*)$, $\theta_j^k$ is the angle between ${\mathcal{P}}_\tau ((G_{\tau}({\bm x}^k))_j)$ and ${\mathcal{P}}_\tau( (G_{\tau}({\bm x}^*))_j)$} (Theorem \ref{thm:angular_convergence}),

\item Combine the available weak convergence result, e.g., \cite{Combettes2004}, with our angular convergence to obtain the convergence in norm on $E$ (Theorem \ref{thm:strong_convergence}).
\end{enumerate}

First, we state without proof a few supporting, known results. From the firm nonexpansiveness of $J_\tau$ and nonexpansiveness of $G_\tau$, following the arguments in \cite[Lemma 4.1]{HYZ08}, we can show that the sets $L$ and $E$ defined above are invariant on $X^*$, thereby justifying the use of Definition \ref{def:L_E_definitions} in studying the convergence of $\{{\bm x}^k\}$ to an arbitrary element ${\bm x}^*\in X^*$.

\begin{lemma}
\label{lem:fp_condition}
Under Assumption \ref{ass:tau_bounds}, for every ${\bm x},{\bm x}'\in X^*$, ${\bm H} {\bm x} = {\bm H} {\bm x}'$, and hence $\nabla \phi_2({\bm x}) = \nabla \phi_2({\bm x}')$.
\end{lemma}

Weak convergence of the forward-backward iterations has been well-established in the general case of monotone inclusion problems, see, e.g., \cite[Section 6]{Combettes2004}. Such problems include \eqref{eq:optimization_problem}, and therefore the weak convergence holds in this setting as well. 

\begin{lemma}
\label{lem:weak_convergence}
Let Assumption \ref{ass:tau_bounds} hold and $\{{\bm x}^{k}\}$ be generated by the forward-backward iterations \eqref{eq:forward_backward_iteration} starting from any ${\bm x}^0\in {\mathcal{H}}$. Then $\{{\bm x}^k\}$ converges weakly to some ${\bm x}^*\in X^*$.
\end{lemma}

\subsection{Finite convergence in the support complement and angular convergence}

In this section, we show two results key in proving strong convergence. The first says that the rows associated with the set $L\subset (\textnormal{supp}({\bm x}^*))^c$ of $\{{\bm x}^k\}$ converge to $\bm{0}$ in finitely many iterations, thus coinciding with ${\bm x}^*$. 
\begin{lemma}
\label{lem:finite_convergence}
Let Assumption \ref{ass:tau_bounds} hold and $\{{\bm x}^{k}\}$ be generated by the forward-backward iterations \eqref{eq:forward_backward_iteration} starting from any ${\bm x}^{0}\in {\mathcal{H}}$. Then
${\bm x}_j^{k} = {\bm x}_j^* = \bm{0}$\; $\forall j\in L$, for all but at most $\|{\bm x}^{0} - {\bm x}^* \|_{2,2}^2/  \omega^2$ iterations.
\end{lemma}

\begin{proof}
The proof follows the arguments in \cite[Lemma 5.2, part 1]{HYZ08}. Let ${\bm x}^*\in X^*$ and $j\in L$. 
Then ${\bm x}_j^* = \bm{0}$ by \eqref{eq:supp_of_solns}, so that $(S_{\tau} ({\bm x}^*))_j = \bm{0}$. Suppose ${\bm x}_j^{k}\neq \bm{0}$ for some $k\in{\mathbb{N}}_0 := \mathbb{N} \cup \{0\}$.
If $\|(G_{\tau} ({\bm x}^{k}))_j \|_2 \le \tau$, then the result follows after $k+1$ iterations, since ${\bm x}^{k+1}_j = (S_{\tau}({\bm x}^k))_j = { (J_{\tau} (G_{\tau} ({\bm x}^k)))_j} = \bm{0}$ by \eqref{eq:J_tau}.
Otherwise, 
$$\|{\bm x}_j^{k+1}\|_2 = \left\| \left(1 - \frac{\tau}{\| (G_{\tau}({\bm x}^k))_j \|_2} \right) (G_{\tau}({\bm x}^k))_j \right\|_2 = \| (G_{\tau} ({\bm x}^{k}))_j  \|_2 - \tau > 0,$$
implying
\begin{align*}
\|{\bm x}_j^{k+1} - {\bm x}_j^* \|_2 
& = \| (G_{\tau} ({\bm x}^{k}))_j - (G_{\tau} ({\bm x}^*))_j +  (G_{\tau} ({\bm x}^*))_j \|_2 - \tau \\
& \leq 
\| (G_{\tau} ({\bm x}^{k}))_j - (G_{\tau} ({\bm x}^*))_j \|_2 - (\tau - \| (G_{\tau} ({\bm x}^*))_j \|_2).
\end{align*}
Hence, by \eqref{eq:omega_alt_def}, 
\begin{gather*}
\begin{aligned}
\|{\bm x}_j^{k+1} - {\bm x}_j^* \|_2^2 
& < \| (G_{\tau} ({\bm x}^{k}))_j - (G_{\tau} ({\bm x}^*))_j \|_2^2 - (\tau - \| (G_{\tau} ({\bm x}^*))_j \|_2)^2 \\
& \leq \| (G_{\tau} ({\bm x}^{k}))_j - (G_{\tau} ({\bm x}^*))_j \|_2^2 - \omega^2. 
\end{aligned}
\end{gather*}
The row-wise nonexpansiveness of $J_{\tau}$ and the nonexpansiveness of $G_\tau$ implies 
\begin{align}
\|{\bm x}^{k+1} - {\bm x}^*\|_{2,2}^2 & \le \|G_{\tau} ({\bm x}^{k}) -  G_{\tau} ({\bm x}^*)\|_{2,2}^2 -  \omega^2 
  \le \|{\bm x}^{k} - {\bm x}^*\|_{2,2}^2 -  \omega^2. \label{eq_v6:1}
\end{align}
Applying \eqref{eq_v6:1} inductively gives $0\le \|{\bm x}^{k+1} - {\bm x}^*\|_{2,2}^2 \le  \|{\bm x}^{0} - {\bm x}^*\|_{2,2}^2 - k  \omega^2$, 
showing that the number of iterations for which ${\bm x}_j^{k} \neq \bm{0}$ for $j\in L$ satisfies $k \leq \|{\bm x}^{0} - {\bm x}^* \|_{2,2}^2/ \omega^2$.
\end{proof}

Let $\sphericalangle ({\bm z}, {\bm z}')$ denote the angle between two nonzero vectors ${\bm z},{\bm z}' \in{\mathcal{H}}^0$, i.e.,
\begin{align*}
\sphericalangle ({\bm z}, {\bm z}') := \cos^{-1} \left( \frac{\langle {\bm z}, {\bm z}' \rangle_2}{\|{\bm z}\|_2 \|{\bm z}'\|_2} \right),
\end{align*}
when both $\bm{z}$ and $\bm{z}'$ are not equal to $\bm{0}\in\mathcal{H}^0$. Our next result shows the angular convergence properties of the forward-backward algorithm using the firmly nonexpansive property. This is a generalization of \cite[Lemma 5.2, part 2]{HYZ08}, which establishes that the signs of the components of gradient steps $G_{\tau}({\bm x}^k)$ (defined in \eqref{eq:G_tau}) agree to those of $G_{\tau}({\bm x}^*)$ for all but finitely many $k\in{\mathbb{N}}$, in case ${\bm x}^k,{\bm x}^* \in {\mathbb{R}}^{N}$.
In essence, the sign function was used to partition ${\mathbb{R}}^N$; and the difference ${\bm x}^{k+1} - {\bm x}^*$ was proved to reduce a fixed amount from the previous step whenever the signs do not match, yielding that the mismatch can only hold for finitely many steps.
Here, we derive an extended argument directly from the firmly nonexpansive property \eqref{eq:J_tauj_FNE}. 
\begin{theorem}
\label{thm:angular_convergence}
Let Assumption \ref{ass:tau_bounds} hold, ${\bm x}^*\in X^*$, and $\{{\bm x}^{k}\}$ be generated by the forward-backward iterations \eqref{eq:forward_backward_iteration} starting from any ${\bm x}^{0}\in{\mathcal{H}}$. Then 
\begin{enumerate}
\item $\sphericalangle({\bm x}^{k}_j,{\bm x}^*_j)\to 0$ as $k\to\infty$ for each $j\in\textnormal{supp}({\bm x}^*)$, 
\item $\sphericalangle({\mathcal{P}}_\tau ((G_{\tau}({\bm x}^k))_j) ,{\mathcal{P}}_{\tau} ( (G_{\tau}({\bm x}^*))_j)) \to 0$ as $k\to\infty$ for each $j\in E\setminus\textnormal{supp}({\bm x}^*)$.
\end{enumerate}
\end{theorem}

\begin{proof}
Let ${\bm z},{\bm z}'\in{\mathcal{H}}$ and $j\in[N]$ be arbitrary. Then ${\mathcal{P}}_\tau ({\bm z}_j) = \kappa\, {\bm z}_j$ and $(J_{\tau}({\bm z}))_j = (1-\kappa) {\bm z}_j$ where $\kappa := \min\{\tau/\|{\bm z}_j\|_2, 1\}\in (0,1]$, from \eqref{eq:J_tau} and \eqref{eq:tau_ball_metric_projection}. Hence 
${\bm z}_j$, ${\mathcal{P}}_\tau({\bm z}_j)$, and $(J_{\tau}({\bm z}))_j$ are all collinear.
We observe that collinearity implies
\begin{gather}
\label{eq_v6:2}
\begin{aligned}
\sphericalangle({\bm z}_j,{\bm z}'_j) &= \sphericalangle ( {\mathcal{P}}_\tau ({\bm z}_j), {\mathcal{P}}_\tau ({\bm z}'_j)), \qquad \forall {\bm z}_j,{\bm z}'_j\ne \bm{0}, \ 
\\
\text{ and } \ \sphericalangle({\bm z}_j,{\bm z}'_j)& = \sphericalangle ((J_{\tau}({\bm z}))_j, (J_{\tau}({\bm z}'))_j),\qquad \forall {\bm z}_j,{\bm z}'_j \not\in B_2(\bm{0},\tau). 
\end{aligned}
\end{gather}

Since $J_\tau$ is component-wise firmly nonexpansive
\begin{align*}
\| \,& {\bm x}_j^{k+1} - {\bm x}_j^*\|_{2}^2 
= \| { (J_{\tau} (G_{\tau} ({\bm x}^{k})))_j} - { (J_{\tau} ( G_{\tau} ({\bm x}^*)))_j} \|_{2}^2 \\
& \le \| (G_{\tau} ({\bm x}^{k}))_j - (G_{\tau} ({\bm x}^*))_j \|_{2}^2 \\
&    \qquad \qquad - \| { ((I - J_{\tau}) G_{\tau} ({\bm x}^{k}))_j } - { ((I - J_{\tau}) G_{\tau} ({\bm x}^*))_j} \|_{2}^2 \\
& = \| (G_{\tau} ({\bm x}^{k}))_j - (G_{\tau} ({\bm x}^*))_j \|_{2}^2 - \| {\mathcal{P}}_\tau ( (G_{\tau} ({\bm x}^{k}))_j) - {\mathcal{P}}_\tau ( (G_{\tau} ({\bm x}^*))_j) \|_{2}^2, \numberthis \label{eq:component-wise_FNE}
\end{align*}
for each $j\in [N]$ and $k\in{\mathbb{N}}_0$. 
Let
\begin{align}
\label{eq:c_j_k}
c_j^{k} & := \| {\mathcal{P}}_\tau ( (G_{\tau} ({\bm x}^{k}))_j) - {\mathcal{P}}_\tau ( (G_{\tau} ({\bm x}^*))_j) \|_{2}^2,
\qquad\qquad    \bar{c}^{k} := \sum_{j\in[N]} c_j^{k}.
\end{align}
From the nonexpansiveness of $G_\tau$, summing \eqref{eq:component-wise_FNE} over $j\in[N]$ and iterating, it follows
\begin{align*}
\|{\bm x}^{k+1} - {\bm x}^* \|_{2,2}^2  & \le \|G_{\tau} ({\bm x}^{k}) - G_{\tau} ({\bm x}^{*}) \|_{2,2}^2 - \bar{c}^{k} \le\| {\bm x}^{k} - {\bm x}^* \|_{2,2}^2 - \bar{c}^{k} \\
 & \le 
 \cdots \le \| {\bm x}^{0} - {\bm x}^* \|_{2,2}^2 - \sum_{\ell = 0}^{k} \bar{c}^{\ell}. \numberthis \label{eq:FNE_iteration_ident} 
\end{align*}
Since $\|{\bm x}^{k+1} - {\bm x}^* \|_{2,2}^2\ge 0$, it follows that $\sum_{\ell = 0}^k \bar{c}^\ell \le \|{\bm x}^0 - {\bm x}^*\|_{2,2}^2$. However, the right hand side is independent of $k$, so that taking limits gives $\sum_{\ell=0}^\infty \bar{c}^\ell \le \|{\bm x}^0 - {\bm x}^*\|_{2,2}^2$. Hence $\bar{c}^k \to 0$ as $k\to\infty$, and, from \eqref{eq:c_j_k}, it follows that for each $j\in[N]$, 
\begin{align}
\label{eq_v6:3}
{\mathcal{P}}_\tau ( (G_{\tau}({\bm x}^k))_j)  \to {\mathcal{P}}_\tau ( (G_{\tau}({\bm x}^*))_j),\ \text{ as }k\to\infty.
\end{align}

 Now, let $j \in  \textnormal{supp}({\bm x}^*)$. Then, from \eqref{eq:J_tau},
\begin{align*}
\bm{0} \neq {\bm x}_j^* = { (J_{\tau}( G_{\tau}({\bm x}^*)))_j} = \frac{ (G_{\tau}({\bm x}^*))_j}{\| (G_{\tau}({\bm x}^*))_j \|_2} \cdot \max \{ \| (G_{\tau}({\bm x}^*))_j \|_2 - \tau , 0 \},
\end{align*}
implying $\| (G_{\tau}({\bm x}^*))_j \|_2 > \tau$, so that $\| {\mathcal{P}}_\tau ( (G_{\tau}({\bm x}^*))_j ) \|_2 = \tau$. For $k\in \mathbb{N}_0$, observe that $\| (G_{\tau}({\bm x}^k))_j \|_2 \le \tau$ yields ${\bm x}_j^{k+1} = (J_{\tau} ( G_{\tau}({\bm x}^k)))_j = \bm{0}$. Therefore, if $\| (G_{\tau}({\bm x}^k))_j \|_2 \le \tau$ holds for infinitely many $k$, there exist infinitely many $k$ such that ${\bm x}_j^{k+1} = \bm{0}$,
contradicting the fact that  ${\bm x}_j^k \rightharpoonup {\bm x}_j^* \ne \bm{0}$ (Lemma \ref{lem:weak_convergence}). This gives $\| (G_{\tau}({\bm x}^k))_j \|_2 > \tau$ for all but finitely many $k\in{\mathbb{N}}_0$.   
Let $K>0$ be such that $\| (G_{\tau}({\bm x}^k))_j \|_2 > \tau$ and ${\bm x}_j^{k+1} \ne \bm{0},\, \forall k\ge K$, we have from \eqref{eq_v6:2} that
\begin{align*}
\sphericalangle({\bm x}_j^{k+1},{\bm x}_j^*) & = \sphericalangle( (J_{\tau} (G_{\tau}({\bm x}^k)))_j , (J_{\tau}( G_{\tau}({\bm x}^*)))_j ) \\
& =  \sphericalangle( {\mathcal{P}}_\tau ( (G_{\tau}({\bm x}^k))_j), {\mathcal{P}}_\tau ( (G_{\tau}({\bm x}^*))_j) ), 
\end{align*}
which, combined with ${\mathcal{P}}_\tau ( (G_{\tau}({\bm x}^k))_j) \to {\mathcal{P}}_\tau ( (G_{\tau}({\bm x}^*))_j )$, implies $\sphericalangle({\bm x}_j^{k},{\bm x}_j^*) \to 0$ as $k\to\infty$.

On the other hand, if $j \in E\setminus \textnormal{supp}({\bm x}^*)$, then ${\bm x}_j^* = \bm{0}$ and from \eqref{eq:L_E_definitions},
\begin{align*}
\| (G_{\tau}({\bm x}^*))_j \|_2 = \|{\bm x}_j^* - \tau (\nabla \phi_{2}({\bm x}^*))_j \|_2 = \tau \| (\nabla \phi_{2}({\bm x}^*))_j  \|_2 = \tau,
\end{align*}
implying  ${\mathcal{P}}_\tau ( (G_{\tau}({\bm x}^*))_j) = (G_{\tau}({\bm x}^*))_j  \ne \bm{0}$.
From \eqref{eq_v6:3}, it follows 
$\sphericalangle( {\mathcal{P}}_\tau ( (G_{\tau}({\bm x}^k))_j),$ ${\mathcal{P}}_{\tau} ( (G_{\tau,j}({\bm x}^*))_j) \to 0 \ \text{ as }k\to\infty. $ 
\end{proof}

\subsection{Strong convergence}

In this section, we establish our main result on strong convergence without strict convexity and compactness assumption. First, we require the following lemma, showing that in general, strong convergence can be implied by weak convergence plus angular convergence.

\begin{lemma}
\label{lem:weak_plus_angular_equals_strong}
Let $\{{\bm z}^k\} \subset {\mathcal{H}}^0$ be such that ${\bm z}^k \rightharpoonup {\bm z}^*$ for some nonzero ${\bm z}^*\in {\mathcal{H}}^0$. Let $\theta^{k} = \sphericalangle ({\bm z}^k, {\bm z}^*)$. If $\theta^{k} \to 0$ as $k\to\infty$, then ${\bm z}^k \to {\bm z}^*$ as $k\to\infty$.
\end{lemma}

\begin{proof}
It is enough to show $\|{\bm z}^k\|_2 \to \|{\bm z}^* \|_2$, since this fact, in combination with ${\bm z}^k \rightharpoonup {\bm z}^*$, implies
\begin{align*}
\|{\bm z}^k - {\bm z}^*\|_2^2 = \|{\bm z}^k\|_2^2 + \|{\bm z}^*\|_2^2 - 2 \langle {\bm z}^k, {\bm z}^* \rangle_2 \to 0 \qquad \text{as } k \to \infty.
\end{align*}
First, the weak convergence gives
\begin{align*}
\|{\bm z}^k\|_2 \cos \theta^{k} = \langle {\bm z}^k , {\bm z}^* \rangle_2/ \|{\bm z}^*\|_2 \to \langle {\bm z}^* , {\bm z}^* \rangle_2 / \|{\bm z}^*\|_2 = \|{\bm z}^*\|_2, \qquad k\to\infty,
\end{align*}
and the angular convergence  $\theta^{k} \to 0$ gives $\cos\theta^{k} \to 1$ as $k\to\infty$. 
On the other hand, the weak convergence also gives $\|{\bm z}^k\|_2 \le M, \ \forall k\in{\mathbb{N}}$ for some $M > 0$. Therefore
\begin{align*}
\|{\bm z}^k\|_2 - \|{\bm z}^*\|_2 = \|{\bm z}^k\|_2 ( 1 - \cos\theta^{k}) + (\|{\bm z}^k\|_2 \cos \theta^{k} - \|{\bm z}^*\|_2) \to 0
\end{align*}
as $k\to\infty$, as desired. 
\end{proof}

With Lemma \ref{lem:weak_plus_angular_equals_strong}, together with the weak and angular convergence established for forward-backward splitting \eqref{eq:forward_backward_iteration} in previous subsections, we are now ready to prove the strong convergence of iterates $\{{\bm x}^k\}$.

\begin{theorem}
\label{thm:strong_convergence}
Let Assumption \ref{ass:tau_bounds} hold and $\{{\bm x}^{k}\}$ be generated by the forward-backward iterations \eqref{eq:forward_backward_iteration} starting from any ${\bm x}^{0}\in {\mathcal{H}}$. Then $\{{\bm x}^{k}\}$ converges strongly to some ${\bm x}^*\in X^*$.
\end{theorem}

\begin{proof}
For $j\in L$, Lemma \ref{lem:finite_convergence} shows that ${\bm x}^{k}_j \to {\bm x}^*_j = \bm{0}$ in a finite number of iterations. 
For $j\in \textnormal{supp}({\bm x}^*)$, Theorem \ref{thm:angular_convergence} shows $\sphericalangle({\bm x}_j^{k},{\bm x}_j^*) \to 0$ as $k\to\infty$. Combining with the weak convergence in Lemma \ref{lem:weak_convergence} and the sufficient condition for strong convergence of Lemma \ref{lem:weak_plus_angular_equals_strong}, this also yields ${\bm x}_j^{k} \to {\bm x}_j^*$.

It remains to consider the case $j \in E\setminus \textnormal{supp}({\bm x}^*)$. First, we have
${\bm x}_j^* = \bm{0}$ and 
$
\| (G_{\tau}({\bm x}^*))_j  \|_2 = \|{\bm x}_j^* - \tau (\nabla \phi_{2}({\bm x}^*))_j  \|_2^2 = \tau,
$ implying $\| {\mathcal{P}}_\tau ( (G_{\tau}({\bm x}^*))_j) \|_2 = \tau$.
Let $\{{\bm x}_j^{k_n}\}$ be the subsequence of all nonzero elements of $\{{\bm x}_j^k\}$, it is enough to show ${\bm x}_j^{k_n} \to \bm{0}$. 
Since ${\bm x}_j^{k_n} \ne \bm{0}\ \forall n$, $\theta^n : = \sphericalangle({\bm x}_j^{k_n}, {\mathcal{P}}_\tau ( (G_{\tau}({\bm x}^*))_j)) $ is well-defined, and from Theorem \ref{thm:angular_convergence},
\begin{align*}
\theta^n  = \sphericalangle( {\mathcal{P}}_\tau ( (G_{\tau}({\bm x}^{k_{n}-1}))_j) , {\mathcal{P}}_{\tau} ( (G_{\tau}({\bm x}^*))_j) ) \to 0, \ \text{ as }n\to \infty,
\end{align*}
implying $\cos \theta^n \to 1$.
From the weak convergence property of ${\bm x}_j^k$,
\begin{align*}
\|{\bm x}_j^{k_n} \|_2 \cos\theta^{n} = \langle {\bm x}_j^{k_n} , {\mathcal{P}}_\tau ( (G_{\tau}({\bm x}^*))_j) \rangle_2/\tau \to \langle {\bm x}_j^*, {\mathcal{P}}_\tau ( (G_{\tau}({\bm x}^*))_j) \rangle_2/\tau = 0 .
\end{align*}
Since $\cos\theta^{n}\to 1$, this gives $\|{\bm x}_j^{k_n} \|_2\to 0$, concluding the proof. 
\end{proof}

\subsection{Linear convergence}

\label{sec:lin_convergence}

In this last subsection, we discuss a sufficient condition to establish linear convergence for the forward-backward splitting method in joint sparse recovery, showing that the path to acquire linear convergence is quite routine given a contractive-type property on $G_\tau$. Recall that the sequence $\{\| {\bm x}^{k} - {\bm x}^* \|_{2,2}\}$ converges to zero $q$-linearly if its $q_1$-factor satisfies
\begin{align*}
q_1 := \limsup_{k\to\infty} \frac{\|{\bm x}^{k+1} - {\bm x}^* \|_{2,2}}{\|{\bm x}^{k} - {\bm x}^* \|_{2,2}} < 1.
\end{align*}
In \cite[Section 4.2]{HYZ08}, under an additional assumption imposing the well-conditioning of a ``reduced" Hessian of $\phi_2$, $q$-linear convergence of the forward-backward splitting method was shown for the single vector recovery problem. With the same assumption (stated in Theorem \ref{thm:fpc_linear_convergence_result}), we are able to establish the linear convergence for our considered joint sparse problem. This assumption aims to make the forward operator $G_\tau$ a contraction. However, as one already has the finite convergence on $L$, roughly speaking, this operator only needs to be a contraction on $E$, explaining why the well-conditioning of a submatrix of ${\bm H}$ associated with $E$ is sufficient. 

For a matrix $\bm{A}$, we define $\lambda_{\max}(\bm{A})$ and $\lambda_{\min}(\bm{A})$ as the maximum and minimum eigenvalues of $\bm{A}$. Let us define
\begin{align}
\label{eq:tau_of_lambda}
\tau(\lambda) := \frac{\gamma(\lambda)}{\gamma(\lambda)+1} \frac{2}{\lambda_{\max}({\bm H})}, \qquad \gamma(\lambda):= \frac{\lambda_{\max}({\bm H})}{\lambda}, \qquad \text{for } \lambda >0.
\end{align}
Then $\tau(\lambda) \in (0,2/\lambda_{\max}({\bm H}))$, since $\gamma(\lambda)>0$.  By ${\bm H}_{EE}$ we define the submatrix of ${\bm H}$ formed by the rows and columns associated with $E$. We have the following theorem, whose proof follows the arguments in \cite[Theorem 4.10]{HYZ08} and is skipped here for brevity. 

\begin{theorem}
\label{thm:fpc_linear_convergence_result}
Let Assumption \ref{ass:tau_bounds} hold, and assume that 
\begin{align}
\lambda_{\min}^E := \lambda_{\min} ({\bm H}_{EE}) > 0.
\end{align}
Then the sequence $\{{\bm x}^{k}\}$ generated by the fixed-point iterations \eqref{eq:forward_backward_iteration} converges to ${\bm x}^*\in X^*$ $q$-linearly. Moreover, if $\tau$ is chosen as in \eqref{eq:tau_of_lambda} with $\lambda = \lambda_{\min}^E$, then the $q_1$-factor satisfies
\begin{align}
\label{eq:q_1_factor_bound}
q_1 \leq \frac{\gamma(\lambda_{\min}^E)-1}{\gamma(\lambda_{\min}^E)+1}.
\end{align}
\end{theorem}

%============================
\section{Concluding remarks}
\label{sec:conclusion}

In this work, we study a forward-backward splitting algorithm for the solution of joint-sparse signal recovery problems, which simultaneously reconstruct a set of sparse signals that are known to share a common sparsity pattern. In such setting, each iteration of the forward-backward algorithm can be viewed as a composition of row-wise soft-thresholding with a step of the standard gradient descent iteration. 
Our analysis shows that this algorithm enjoys the similar strong convergence property that has been shown in single vector recovery \cite{HYZ08}, even in the case that the sets of measurements and of signals to reconstruct are infinite. The major theoretical contribution of this paper, therefore, is a proof of strong convergence of forward-backward splitting method without strict convexity and compactness assumptions. 

Applications which fit this model arise in imaging, data analysis, sensor arrays, and the approximation of high-dimensional parameterized PDEs. In solving parameterized PDEs, when combined with the standard compressed sensing scheme for polynomial approximation, the benefits of joint sparse approach are manifold. First, it enables simultaneous, global (rather than pointwise) approximations of the solution in the physical space. As such, this approach exploits the joint sparsity structure and provably requires fewer samples than in the case of reconstructing multiple single vectors, as demonstrate in \cite{EldarRauhut10,DaviesEldar12}. 
In addition, joint-sparse recovery of the PDE solutions relies on the decay of the polynomial coefficients and {a priori} estimates of the truncation error in global energy norms, which are well established in the existing literature, see, e.g., \cite{CDS11,CCS14,TranWebsterZhang14}. These advantages make joint-sparse approach an attractive alternative for the solution of high-dimensional parameterized PDEs; and we have documented this study in \cite{DexterTranWebster17}.

%============================
\paragraph{Acknowledgements}
The first author acknowledges the support of the Pacific Institute of Mathematical Sciences (PIMS).
The second and third authors acknowledge support from: the U.S. Department of Energy, Office of Science, Office of Advanced Scientific Computing Research, Applied Mathematics program under contracts and awards ERKJ314, ERKJ331, ERKJ345, and Scientific Discovery through Advanced Computing (SciDAC) program through the FASTMath Institute under Contract No. DE-AC02-05CH11231; and by the Laboratory Directed Research and Development program at the Oak Ridge National Laboratory, which is operated by UT-Battelle, LLC., for the U.S. Department of Energy under contract DE-AC05-00OR22725.

%References
\bibliographystyle{amsplain}
\bibliography{SVAA_refs} 

\end{document}